\documentclass{article}

\usepackage{amsmath, amsthm, amssymb}
\usepackage{color,xcolor}
\newtheorem{lemma}{Lemma}%[section]

\newtheorem{proposition}[lemma]{Proposition}
\newtheorem{remark}[lemma]{Remark}

\newtheorem{theorem}[lemma]{Theorem}
\newtheorem{definition}[lemma]{Definition}

\newcommand{\e}{\varepsilon} 

\title{A Bourgain-Brezis-Mironescu result\\ for fractional thin films}
\author{Andrea Braides\footnote{Department of Mathematics, University of Rome Tor Vergata, via della ricerca scientifica 1, 00133 Rome, Italy, {\tt braides@mat.uniroma2.it}} \ 
and Margherita Solci\footnote{DADU, Universit\`a di Sassari, piazza Duomo 6, 07041 Alghero, Italy
}}
\date{}
\begin{document} 

\maketitle

\begin{abstract} We consider the limit of squared $H^s$-Gagliardo seminorms on thin domains of the form $\Omega_\varepsilon=\omega\times(0,\varepsilon)$ in $\mathbb R^d$.
When $\varepsilon$ is fixed, multiplying by $1-s$ such seminorms have been proved to converge as $s\to 1^-$ to a dimensional constant $c_d$ times the Dirichlet integral on $\Omega_\varepsilon$ by Bourgain, Brezis and Mironescu. In its turn such Dirichlet integrals divided by $\varepsilon$ converge as $\varepsilon\to 0$ to a dimensionally reduced Dirichlet integral on $\omega$. We prove that if we let simultaneously $\varepsilon\to 0$ and $s\to 1$ then these squared seminorms still converge to the same dimensionally reduced limit when multiplied by $(1-s) \varepsilon^{2s-3}$, independently of the relative converge speed of $s$ and $\varepsilon$. This coefficient combines the geometrical scaling $\varepsilon^{-1}$ and the fact that relevant interactions for the $H^s$-Gagliardo seminorms are those at scale $\varepsilon$.
We also study the usual membrane scaling, obtained by multiplying by $(1-s)\varepsilon^{-1}$, which highlighs the {\em critical scaling} $1-s\sim|\log\varepsilon|^{-1}$, and the limit when $\varepsilon\to 0$ at fixed $s$.\end{abstract}

\section{Introduction}
We consider a fractional non-local analog of the variational theory of thin films as 
studied for example in \cite{LDR,BFF} for integral functionals. 
In the local case, one considers a thin domain $\Omega_\e$ 
in $\mathbb R^d$ of the form $\omega\times (0,\e)$ with $\omega\subset\mathbb R^{d-1}$, and describes the 
asymptotic properties, as $\e\to 0$, of minimizers of energies
\begin{equation}\label{1}
\int_{\Omega_\e} W(\nabla u)\,dx
\end{equation}
subjected to suitable boundary conditions on the lateral boundaries 
$(\partial\omega)\times (0,\e)$ and scaled applied forces. 
Besides its own interest in the asymptotic description of thin object,
our interest in the subject is also motivated by a discussion with F.~Murat,
who observed that the difficulty in describing some regimes in the homogenization
of fractional energies (see \cite{BBD}) may be due to their behaviour on some sets that
of thin-film type.

We briefly recall how the local case can be treated.
After scaling energies \eqref{1} as
\begin{equation}\label{1bis}
\frac1\e\int_{\Omega_\e} W(\nabla u)\,dx
\end{equation}
and defining a {\em dimension-reduction convergence} of functions $u_\e$ defined on $\Omega_\e$ to a function $u$ defined on $\omega$, this problem can be recast as the
description of the $\Gamma$-limit of these scaled energies with respect to that convergence. The $\Gamma$-limit results to be described
as a {\em dimensionally reduced} energy of the form
\begin{equation}
\int_{\omega} W_{\rm dr}(\nabla' u)\,dx'
\end{equation}
defined on the $d-1$-dimensional set $\omega$ for suitable 
$W_{\rm dr}$; the apex denotes $d-1$-dimensional quantities.
In the convex case $W_{\rm dr}$ is simply obtained 
by minimizing out the dependence of $W$ on the $d$-th derivative, but 
in the non-convex vector case its characterization involves relaxation and
homogenization arguments (see e.g.~\cite{LDR,BFF}).

The key functional argument in the reasoning just illustrated amounts to find the first ``critical scaling'' $\frac1\e$ so that the scaled 
energies \eqref{1bis} are equi-coercive with respect to dimension-reduction convergence.
Note that this compactness argument depends only on the super-linear growth of $W$, so that
we may take 
$$
\frac1\e\int_{\Omega_\e} |\nabla u|^2\,dx
$$
as a model. This is only done for ease of notation, the case $p>1$ with $p\neq 2$ being completely analogous.

In a non-local model, we may consider as a prototype, in the place of \eqref{1}, the fractional quadratic energies 
\begin{equation}\label{Fes-def}
F_{\e,s}(u)=(1-s) \int_{\Omega_\e\times \Omega_\e} \frac{|u(x)-u(y)|^2}{|x-y|^{d+2s}} dx\,dy,
\end{equation}
defined on the space $H^s(\Omega_\e)$ with $s\in(0,1)$.
By the result of Bourgain, Brezis, and Mironescu \cite{BBM}, letting $s\to 1$
these energies approximate the energy in \eqref{1} with $W(\nabla u)=
c_d|\nabla u|^2$, where 
\begin{equation}
c_d=\frac1{2d}\mathcal H^{d-1}(S^{d-1})
\end{equation} depending on $d$. Hence,
the corresponding dimensionally reduced energy is of the form
\begin{equation}\label{c_d}
F_{\rm dr}(u)= c_d\int_\omega |\nabla' u|^2\,dx'.
\end{equation}
This limit is obtained by letting first $s\to 1$ in \eqref{Fes-def}, dividing the result by $\e$, and then letting 
$\e\to 0$, noting at the same time that the limit is finite only for $u$  not depending on
the `vertical' variable $x_d$. 
%In this paper we address the problem of finding the correct scaling $\lambda_{\e,s}$
%such that the energies 
%\begin{equation}\label{Fes-def}
%F_{\e,s}(u)=\frac{1-s}{\lambda_{\e,s}} \int_{\Omega_\e\times \Omega_\e} \frac{|u(x)-u(y)|^2}{|x-y|^{d+2s}} dx\,dy,
%\end{equation}
%admit a dimensionally-reduced energy as $\e\to 0^+$ and $s=s_\e\to 1^-$, and describe this limit energy. 

\smallskip
In this paper we address the problem of finding the scaling $\lambda_{\e,s}$ for which the functionals 
$\lambda_{\e,s} F_{\e,s}$ are equicoercive with respect to a properly defined dimension-reduction convergence, and compute the dimensionally-reduced limit if $\e\to 0^+$ and $s=s_\e\to 1^-$. 
We find that the correct scaling factor is $\lambda_{\e,s}=\e^{2s-3}$. This scaling can be explained by testing the pointwise convergence of $F_{\e,s}(u)$ when the target function is of the form $u(x)=u(x',x_d)=u(x')$.  Indeed, if $u$ is smooth then the double integral in \eqref{Fes-def} can be restricted to pairs satisfying $|x-y|<\e$, for which the term $u(y)-u(x)$ can be replaced by $\langle \nabla u(x),y-x\rangle$. For such $u$ we have, remembering that $u=u(x')$,
\begin{eqnarray*}
F_{\e,s}(u)&\sim& (1-s) \int_{\Omega_\e} \int_{B_\e(x)}\frac{|\langle \nabla u(x),y-x\rangle|^2}{|x-y|^{d+2s}} dx\,dy
\\
&\sim& (1-s) \frac1d\int_{\Omega_\e} |\nabla u(x)|^2 \,dx \int_{B_\e}\frac{1}{|\xi|^{d+2s-2}} d\xi
\\
&\sim& (1-s) \frac\e{d}\int_{\omega} |\nabla' u(x')|^2 \,dx'\  \mathcal H^{d-1}(S^{d-1})\int_{0}^{\e}{t^{1-2s}} dt
\\
&=& (1-s) \frac\e{d}\int_{\omega} |\nabla' u(x')|^2 \,dx'\  \mathcal H^{d-1}(S^{d-1})\frac{\e^{2-2s}}{2(1-s)} 
\\
&=& \e^{3-2s} c_d\int_{\omega} |\nabla' u(x')|^2 \,dx', 
\end{eqnarray*} 
which gives that the (pointwise) limit of $\e^{2s-3}F_{\e,s}(u)$ is given by \eqref{c_d}.

This argument shows that, contrary to the local case, the factor $\lambda_{\e,s}$ is not purely due to the geometric dimension $\e$, but also takes into account that relevant interactions in the double integral in \eqref{Fes-def} are those with $|x-y|<\e$, which give an extra $\e^{2(1-s)}$. As regards the proof of a compactness result for di\-men\-sion-reduction fractional convergence, it is worth noting that the non-local nature of the Gagliardo seminorm makes it difficult to use scaling arguments in the $d$-th direction as those used in the local case, and we have to resort to a different  argument by discretization. An interesting fact is that the resulting energy after scaling is the $d$-dimensional dimensionally reduced functional in \eqref{c_d}, even for `very thin films', for which some sort of $d-1$-dimensional behaviour could be expected. A heuristic explanation why this does not happen is that, unlike the local case, the Gagliardo seminorm defining a fractional Sobolev space has a kernel with a homogeneity depending on the dimension, which is in a sense incompatible with dimension reduction.

 A different, more usual scaling, is the usual membrane scaling; that is, the one obtained dividing $F_{\e,s}$ by $\e$. This scaling allow to determine a {\em critical scaling}, when 
$$
1-s\sim \frac1{|\log(\e)|};
$$
more precisely, when $\e^{1-s}\to \kappa$, for which the limit is the functional  in \eqref{c_d} multiplied by $\kappa^2$. The corresponding {\em subcritical} and {\em supercritical} regimes correspond to the separation of scales described above, and the case when formally $\e\to 0$ first and then  $s\to 1^-$. In the latter case, the limit is $0$. Note that in the case of a trivial limit; that is, when $\kappa=0$ the scale $\e^{3-2s}$ can be interpreted as the next term in an expansion by $\Gamma$-convergence in the sense of \cite{BT}.

\section{Notation and preliminaries}
In the following, $\omega$ is a bounded connected open subset of $\mathbb R^{d-1}$ and for all $\e>0$
we define the {\em thin film}
$$
\Omega_\e=\omega\times (0,\e)\subset \mathbb R^d.
$$
If $x\in\mathbb R^d$, then we write $x'=(x_1,\ldots, x_{d-1})$, and use the notation $x=(x',x_d)$.
We also write 
$$
\nabla'u=\Big(\frac{\partial u}{\partial x_1},\ldots, \frac{\partial u}{\partial x_{d-1}}\Big).
$$
With a slight abuse of notation, this is done both when $u=u(x')$ and $u=u(x)$. In the second case, we also write the usual gradient as $\nabla u=(\nabla'u, \tfrac{\partial u}{\partial x_d})$.

\subsection{Dimension-reduction convergence}
We first give a notion of convergence of functions $u_\e\in L^1(\Omega_\e)$ to a dimensionally reduced parameter $u\in L^1(\omega)$. It is customary to do this by scaling as follows \cite{LDR,BFF}.

\begin{definition}[dimension-reduction convergence]
We say that $u_\e\in L^1(\Omega_\e)$ {\em (di\-men\-sion-reduction) converge} to $u\in L^1(\omega)$ as $\e\to0$, and we simply use the notation $u_\e\to u$, if $u$ is defined by the following procedure.
 
{\rm1)} The functions $u_\e$ are scaled to a common space by defining $v_\e\in L^1(\Omega)$, where $\Omega=\omega\times(0,1)$, by $v_\e(x)=v_\e(x',x_d)= u_\e(x',\e x_d)$;

{\rm2)} we have the convergence of $v_\e$ to $v$ in $L^1_{\rm loc}(\Omega)$ to some $v=v(x')$; that is, to some $v$  is independent of $x_d$;

{\rm3)} we define the limit $u\in L^1(\omega)$ by the equality $u(x')=v(x')$. 
\end{definition}

In the case of thin films modeled by local energies on Sobolev spaces, this convergence is justified by the following easy compactness result (see for example \cite{GCB} Section 14.1).

\begin{lemma}[(local) dimension-reduction compactness]\label{local-lemma}
Let $u_\e\in H^1(\Omega_\e)$ and suppose that 
$$
\sup_\e\frac1\e\int_{\Omega_\e} |\nabla u_\e|^2dx<+\infty.
$$
Then, up to addition of constants, $u_\e$ is precompact (that is, there exists $c_\e$ such that $u_\e+c_\e$ is precompact) with respect to the convergence above, the limit  $u$ belongs to $H^1(\omega)$ and
$$
\int_\omega |\nabla'u|^2dx'\le \liminf_{\e\to 0}\frac1\e\int_{\Omega_\e} |\nabla u_\e|^2dx.
$$ 
\end{lemma}

\subsection{Fractional energies and their limit as $s\to1$}
If $\Omega\subset \mathbb R^d$ is a bounded connected open set, the fractional Sobolev spaces $H^s(\Omega)$ are defined as the set of functions in $L^2(\Omega)$ such that their {\em Gagliardo seminorm}
$$
[u]_{H^{s}(\Omega)}=\bigg(\int_{\Omega\times \Omega} \frac{|u(x)-u(y)|^2}{|x-y|^{d+2s}} dx\,dy\bigg)^{1/2}
$$
is finite (see \cite{leofrac,DPV}. 

The space $H^1(\Omega)$ is a singular limit of the spaces $H^s(\Omega)$ in the following sense.

\begin{theorem}[Bourgain--Brezis--Mironescu limit theorem]
If $u_s$ is a family of functions with $u_s\in H^s(\Omega)$ and $\sup_s(1-s)[u]_{H^s(\Omega)}<+\infty$,
then, up to subsequences and addition of constants, $u_s$ converges in $L^2(\Omega)$ as $s\to 1$ to a function $u\in H^1(\Omega)$. Furthermore, for $u\in H^1(\Omega)$ we have
$$
\Gamma\hbox{-}\lim_{s\to 1} (1-s)\int_{\Omega\times \Omega} \frac{|u(x)-u(y)|^2}{|x-y|^{d+2s}} dx\,dy=c_d\int_\Omega|\nabla u|^2dx,
$$
with $\displaystyle c_d=\frac{\mathcal H^{d-1}(S^{d-1})}{2d}$.
\end{theorem}

Contrary to the local case, the form of the Gagliardo seminorm is dependent on the dimension, so that the same expression may have different implications. In particular, we will use the following characterization of constant functions (see \cite[Proposition 2]{brezis2002recognize}).

\begin{proposition}\label{constant}
Let $\Omega$ be a connected open subset of $\mathbb R^k$ and $u:\Omega\to \mathbb R$ is a measurable function such that
$$
\int_{\Omega\times \Omega} \frac{|u(x)-u(y)|^2}{|x-y|^{k+2}} dx\,dy<+\infty,
$$
then $u$ is constant.
\end{proposition}

\section{Scaling regimes of fractional thin films}
In this section we compute the pointwise limit on dimensionally reduced (smooth) functions of the functionals $F_{\e,s}$ defined in \eqref{Fes-def}. More precisely, we prove the following statement.

\begin{proposition}\label{limsupc2} Let $u\in C^2(\overline\omega)$, and with an abuse of notation, let $u$ also denote the function $u=u(x')$, independent of the $d$-th variable, which we view as an element of  $H^s(\Omega_\e)$. Then we have
\begin{equation}
\lim_{\e\to 0} \frac{1}{\e^{3-2s}}F_{\e,s}(u)= c_d\int_\omega |\nabla'u|^2dx'
\end{equation}
for all $s=s_\e$ with $s_\e\to 1^-$ as $\e\to 0^+$.
\end{proposition}

\begin{proof}
We first note that 
\begin{equation}\label{8}
\limsup_{\e\to 0} \frac{1-s}{\e^{3-2s}} \int_{\{(x,y)\in\Omega_\e: |x-y|>r\e\}}\frac{|u(x)-u(y)|^2}{|x-y|^{d+2s}} dx\,dy=0
\end{equation}
for all $r>0$.
Indeed, if $L$ is such that $|u(x)-u(y)|\le L|x-y|$, then we have
\begin{eqnarray*}&&
\hskip-1cm\int_{\{(x,y)\in\Omega_\e:\times\Omega_\e |x-y|>r\e\}}\frac{|u(x)-u(y)|^2}{|x-y|^{d+2s}} dx\,dy\\
&\le& \int_{\{(x,y)\in\Omega_\e\times\Omega_\e: |x-y|>r\e\}}L^2|x-y|^{2-d-2s} dx\,dy\\
&\le& C\Big(\int_{\Omega_\e} \int_{\{r\e<|\xi|<2\e\}}|\xi|^{2-d-2s} d\xi\,dy\\
&&+
\int_{\Omega_\e}\int_{\{y\in\Omega_\e: |x'-y'|>\e\}}|x'-y'|^{2-d-2s} dx\,dy\Big)\\\\
&\le& C\Big( \mathcal H^{d-1}(S^{d-1})\e|\omega| \int_{r\e}^{2\e}t^{1-2s} dt +
\mathcal H^{d-2}(S^{d-2})\e^2 |\omega| \int_{2\e}^{\infty}t^{-2s}dt\Big)
\\
&\le& C\Big(\frac{\e}{2(1-s)}((r\e)^{2-2s}-(2\e)^{2-2s})+\frac{\e^2}{2s-1}(2\e)^{1-2s}\Big)\\
&\le& C\Big( \frac{\e^{3-2s}}{1-s}(r^{2-2s}-2^{2-2s})+ \e^{3-2s}\Big).
\end{eqnarray*}
Hence, we have
\begin{equation}\label{a1}
\frac{1-s}{\e^{3-2s}} \int_{\{(x,y)\in\Omega_\e: |x-y|>r\e\}}\frac{|u(x)-u(y)|^2}{|x-y|^{d+2s}} dx\,dy\le C\big(r^{2-2s}-2^{2-2s}+ 1-s\big)
\end{equation}
Letting $s\to 1$, we have \eqref{8}.

From \eqref{8}, we obtain that the asymptotic behaviour of $\frac{1}{\e^{2-2s}}F_{\e,s}(u)$ is the same as that
of
$$
\frac{1-s}{\e^{3-2s}} \int_{\{(x,y)\in\Omega_\e: |x-y|<r\e\}}\frac{|u(x)-u(y)|^2}{|x-y|^{d+2s}} dx\,dy
$$
with truncated range of interactions.

We now simplify the asymptotic analysis when $|x-y|<r\e$. We can write
$$
u(x)-u(y)= \langle \nabla u(x), x-y\rangle + O(|x-y|^2)
$$
uniformly in $x$, so that, with fixed $\eta>0$
\begin{eqnarray*}
||u(x)-u(y)|^2- |\langle \nabla u(x), x-y\rangle|^2|&\le& \eta|\langle \nabla u(x), x-y\rangle|^2 +C_\eta |x-y|^4\\
&\le& \eta C|x-y|^2 +C_\eta |x-y|^4,
\end{eqnarray*}
and
\begin{eqnarray}\label{a2}\nonumber
&&\hskip-1.5cm\bigg|\int_{\{(x,y)\in\Omega_\e: |x-y|<r\e\}}\frac{|u(x)-u(y)|^2}{|x-y|^{d+2s}} dx\,dy
\\ \nonumber
&& -\int_{\{(x,y)\in\Omega_\e: |x-y|<r\e\}}\frac{|\langle \nabla u(x), x-y\rangle|^2}{|x-y|^{d+2s}} dx\,dy\bigg|\\ \nonumber
&\le&  \eta C\e|\omega|\int_{|\xi|<r\e\}}{|\xi|^{2-d-2s}} d\xi+C_\eta \e|\omega|\int_{|\xi|<r\e\}}{|\xi|^{4-d-2s}} dx\,dy\\\nonumber
&\le&  C\Big( \eta \frac{1}{1-s}\e^{3-2s}+C_\eta \e^{5-2s}\Big)\\
&=&  C\frac{1}{1-s}\e^{3-2s}\Big(  \eta+ C_\eta (1-s)\e^2\Big).
\end{eqnarray}
Letting $\e\to 0^+$  first, by the arbitrariness of $\eta$ and this estimate, together with \eqref{8}, we also have that the asymptotic behaviour of ${\e^{2s-3}}F_{\e,s}(u)$ is the same as that
of
$$
\frac{1-s}{\e^{3-2s}} \int_{\{(x,y)\in\Omega_\e: |x-y|<r\e\}}\frac{|\langle \nabla u(x), x-y\rangle|^2}{|x-y|^{d+2s}} dx\,dy.
$$

We now take $r<\frac12$, so that 
\begin{eqnarray*}
&&\hskip-1cm\int_{\{(x,y)\in\Omega_\e: |x-y|<r\e\}}\frac{|\langle \nabla u(x), x-y\rangle|^2}{|x-y|^{d+2s}} dx\,dy\\
&\ge &\int_{\omega\times (r\e,(1-r)\e)}\int_{B_{r\e}(x)}\frac{|\langle \nabla u(x), x-y\rangle|^2}{|x-y|^{d+2s}}\,dy\,dx\\
&=&\int_{\omega\times (r\e,(1-r)\e)}\int_{B_{r\e}}\frac{|\langle \nabla u(x), \xi\rangle|^2}{|\xi|^{d+2s}}\,d\xi\,dx\\
&=&\int_{\omega\times (r\e,(1-r)\e)}|\nabla u(x)|^2\frac1d\int_{B_{r\e}}|\xi|^{2-d-2s}\,d\xi\,dx\\
&=&\int_{\omega\times (r\e,(1-r)\e)}|\nabla u(x)|^2c_d(r\e)^{2-2s}\,dx\\
&=&(1-2r)r^{2-2s}\frac{\e^{3-2s}}{1-s}c_d\int_{\omega}|\nabla' u(x')|^2dx'.
\end{eqnarray*}
Between the third and fourth line of the previous formula, we have used the remark that, by the symmetry of the domain of integration, we have 
\begin{eqnarray*}
\int_{B_{r\e}}\frac{|\langle \nabla u(x), \xi\rangle|^2}{|\xi|^{d+2s}}\,d\xi&=&|\nabla u(x)|^2 \int_{B_{r\e}}\frac{|\langle e_j, \xi\rangle|^2}{|\xi|^{d+2s}}\,d\xi\\
&=&|\nabla u(x)|^2 \frac1d\int_{B_{r\e}}\frac{|\xi|^2}{|\xi|^{d+2s}}\,d\xi
\end{eqnarray*}
for all elements of the canonical basis $\{e_1,\ldots, e_d\}$.
Hence,
\begin{equation}\label{9}
\liminf_{\e\to 0}\frac{1}{\e^{3-2s}}F_{\e,s}(u)\ge (1-2r)c_d\int_{\omega}|\nabla' u(x')|^2dx'
\end{equation}
for all $r<\frac12$. Conversely, for all $r>0$ we have
\begin{eqnarray}\label{a3}\nonumber
&&\hskip-1cm\int_{\{(x,y)\in\Omega_\e: |x-y|<r\e\}}\frac{|\langle \nabla u(x), x-y\rangle|^2}{|x-y|^{d+2s}} dx\,dy\\ \nonumber
&\le &\int_{\omega\times (0,\e)}\int_{B_{r\e}(x)}\frac{|\langle \nabla u(x), x-y\rangle|^2}{|x-y|^{d+2s}}\,dy\,dx
\\
&=&r^{2-2s}\frac{\e^{3-2s}}{1-s}c_d\int_{\omega}|\nabla' u(x')|^2dx',
\end{eqnarray}
repeating the same computations as above, so that 
\begin{equation}\label{10}
\limsup_{\e\to 0}\frac{1}{\e^{3-2s}}F_{\e,s}(u)\le c_d\int_{\omega}|\nabla' u(x')|^2dx'.
\end{equation}
The claim follows from \eqref{9} and \eqref{10} by letting $r\to 0$.
\end{proof}

\begin{remark}\label{rem-nd}\rm
From \eqref{a1}, \eqref{a2} and \eqref{a3}, we obtain that $F_{\e,s}(u)\le\e^{3-2s} C$ for all $s$, with $C$ depending on $u$ but independent of $s$. In particular, we have
\begin{equation}\label{a4}
\lim_{\e\to 0} \lambda_{\e,s}F_{\e,s}(u)=0
\end{equation}
if $\lambda_{\e,s}=o(\e^{3-2s})$, independently whether $s\to 1$ or not.
\end{remark}

\section{Dimension-reduction convergence and compactness}

In this section we prove a $\Gamma$-convergence result when the energies are scaled as in the previous section.

\subsection{Discretization of Gagliardo seminorms}
We first note that the proof of the compactness result in Lemma \ref{local-lemma}
heavily relies on that fact that we may write
\begin{eqnarray*}
\frac1\e\int_{\Omega_\e} |\nabla u_\e|^2dx&=&\int_{\omega\times(0,1)}|\nabla' v_\e|^2dx+\frac1{\e^2} \int_{\omega\times(0,1)}\Big|\frac{\partial v_\e}{\partial x_d}\Big|^2dx\\
&\ge& \int_{\omega\times(0,1)}|\nabla v_\e|^2dx,
\end{eqnarray*}
which at the same time proves compactness for $v_\e$ in $H^1(\omega\times (0,1))$ and that $\frac{\partial v_\e}{\partial x_d}$ tends to $0$.
Unfortunately, this decoupling of the `horizontal' and `vertical' derivatives is not immediate for Gagliardo seminorms.
In order to bypass this difficulty we will use a discretization argument coupled with Lemma \ref{local-lemma}.

\smallskip
Following the notation in \cite{Solci-vortices} (see also \cite{BBD}) we define the {\em set of orthonormal bases} (Stiefel manifold) of $\mathbb{R}^d$
\[
V:=\{\overline{\nu}=(\nu_1,...,\nu_d) : \nu_j \in S^{d-1} \text{ such that } \langle\nu_i, \nu_j\rangle=0 \text{ for } i\neq j \}
\]
and observe that $V$ has Hausdorff dimension equal to $k_d:=d(d-1)/2$. 

Given $\rho>0$ and $\overline{\nu}\in V$, we define $$\mathbb Z^d_{\rho\overline{\nu}}:=\{z_1\rho\nu_1+z_2\rho\nu_2+...+z_d\rho\nu_d : (z_1,...,z_d)\in \mathbb Z^d\}$$ and $Q_{\rho\overline{\nu}}$ as the cube described by the orthogonal basis $\{\rho\nu_1,...,\rho\nu_d\}$.

\smallskip
We describe a discretization procedure as follows. Given $\Omega$ an open set in $\mathbb R^d$, for $\overline{\nu}\in V$, $\rho>0$, and $r>0$ we set
\[
\mathcal{I}^r_{\rho\overline{\nu}}(\Omega):=\{k\in \mathbb Z_{\rho\overline{\nu}}^d : rk+rQ_{\rho\overline{\nu}} \subset \subset \Omega\},
\]
and note that for every $\nu\in S^{d-1}$ and  for every $\overline{\nu}\in V$ it holds
\[
\bigcup_{k\in \mathcal{I}^r_{\rho \overline{\nu}}(\Omega)} rk+rQ_{\rho\overline{\nu}} \subseteq \{x\in \Omega : x+r\rho\nu\in\Omega\}.
\]

Given  $r>0$, $\rho\in (0,1)$ and $\overline{\nu}\in V$, for all $u\in L^1(\Omega)$ we define the function $u^{r,\rho\overline{\nu}}$ in two steps:

    (i) first, we assign values on the lattice $\mathcal{I}^r_{\rho\overline{\nu}}$ setting
    \[
    u^{r,\rho\overline{\nu}}(rk)= \displaystyle \frac{1}{|r\rho|^d}\int_{rk+rQ_{\rho\overline{\nu}}} u\,dx \qquad \hbox{ for every } k\in \mathcal{I}^r_{\rho\overline{\nu}};
    \] 
    
    (ii) then, given $k$ in the `interior' of  $\mathcal{I}^r_{\rho\overline{\nu}}(\Omega)$ defined as
    \[
{\mathring{\mathcal I}}^r_{\rho\overline{\nu}}(\Omega):=\{k\in \mathbb Z_{\rho\overline{\nu}}^d : rk+2rQ_{\rho\overline{\nu}} \subset \subset \Omega\},
\]
consider the cube $rk+rQ_{\rho\overline{\nu}}$, $\tau$ 
a permutation of the indices $\{1,...,d\}$, and $rk+r\Delta^\tau_{\rho\overline{\nu}}$ the corresponding simplex in Kuhn's decomposition with vertices $rk, rk+r\Delta^{\tau,0}_{\rho\overline{\nu}},rk+r\Delta^{\tau, 1}_{\rho\overline{\nu}},\ldots, rk+r\Delta^{\tau, d}_{\rho\overline{\nu}}$ (see \cite[Lemma 1]{Kuhn1960}), on such a simplex we define $u^{r,\rho\overline{\nu}}$ as the affine interpolation of the previously defined values $u^{r,\rho \overline{\nu}}(rk), u^{r,\rho \overline{\nu}}(rk+r\Delta^{\tau,0}_{\rho\overline{\nu}}),\, u^{r,\rho\overline{\nu}}(rk+r\Delta^{\tau, 1}_{\rho\overline{\nu}}),..., \,u^{r,\rho\overline{\nu}}(rk+r\Delta^{\tau, d}_{\rho\overline{\nu}})$.

\begin{lemma}\label{lemma-1}
 If $u\in H^s(\Omega)$, we have 
\begin{eqnarray*} &&\hskip-1cm
\int_{\Omega\times\Omega} \frac{|u(x)-u(y)|^2}{|x-y|^{d+2s}} dx\,dy \\
&& \ge \frac{r^{2(1-s)}}{d} \frac{\mathcal H^{d-1}(S^{d-1})}{\mathcal{H}^{k_d}(V)} \int_0^1\int_V \int_{\Omega'}  |\nabla u^{r,\rho\overline{\nu}}|^2\,dx\,   d\mathcal{H}^{k_d}(\overline{\nu})\,\rho^{1-2s}\, d\rho,
\end{eqnarray*}
where $\Omega'$ is any open subset contained in $\bigcup_{k \in {\mathring{\mathcal I}}^r_{\rho\overline{\nu}}(\Omega)} rk+rQ_{\rho\overline{\nu}}$ for all $\overline \nu$ and $\rho\in(0,1]$.
\end{lemma}

\begin{proof}
The proof is contained in the first part of \cite[Section 3.1]{BBD}. The claim of the lemma corresponds to (27) therein, taking the coefficient $a$ in that formula identically equal to $1$. 
\end{proof} 

We then obtain the following intermediate estimate.
\begin{proposition}\label{prop4}
Let $u_\e\in H^s(\Omega_\e)$, let $\sigma\in (0,1)$ and define
\begin{equation}\label{sigmae}
u^\sigma_\e(x):=\frac{2(1-s)}{\mathcal{H}^{k_d}(V)} \int_{[0,1]\times V}u^{\sigma\e,\rho\overline{\nu}}_\e(x)\rho^{1-2s}  d\mathcal{H}^{k_d}.
\end{equation}
Then $u^\sigma_\e\in H^1(\Omega_\e)$ and 
\begin{equation}\label{est-2}
c_d\sigma^{2(1-s)}\frac1\e \int_{(\Omega')_\e} |\nabla u^\sigma_\e|^2dx\le \e^{2s-3} F_\e(u_\e)
\end{equation}
for $\e$ small enough for each $\Omega'$ compactly contained in $\omega\times(\sigma,1-\sigma)$,
where
$$
(\Omega')_\e=\big\{(x',x_d): \big(x',\tfrac1\e x_d\big)\in \Omega'\big\}.
$$
\end{proposition}

\begin{proof} We note that for $\e$ small $(\Omega')_\e$ is contained in $\bigcup_{k \in {\mathring{\mathcal I}}^{\e\sigma}_{\rho\overline{\nu}}(\Omega)} \e\sigma k+\e\sigma Q_{\rho\overline{\nu}}$ for all $\overline \nu$ and $\rho\in(0,1]$. Hence, we can apply Lemma \ref{lemma-1} with $\Omega_\e$ in the place of $\Omega$, $r=\e\sigma$ and $(\Omega')_\e$ in the place of $\Omega'$. We then obtain
\begin{eqnarray*}
\e^{2s-3} F_\e(u_\e)&=&\frac{1-s}{\e^{3-2s}}
\int_{\Omega_\e\times\Omega_\e} \frac{|u_\e(x)-u_\e(y)|^2}{|x-y|^{d+2s}} dx\,dy \\
& \ge& \frac{\sigma^{2(1-s)}}{2d\e} \mathcal H^{d-1}(S^{d-1}) \int_{[0,1]\times V} \int_{(\Omega')_\e}  |\nabla u_\e^{\e\sigma,\rho\overline{\nu}}|^2\,dx\, d\mu_\e(\rho,\overline \nu),\\
\end{eqnarray*}
where 
$$
d\mu_\e(\rho,\overline \nu)= 2(1-s)
 \frac{\rho^{1-2s}}{\mathcal{H}^{k_d}(V)}\, d\rho\,d\mathcal{H}^{k_d}(\overline{\nu})
 $$ 
 gives a probability measure on $[0,1]\times V$. The claim then follows by applying Jensen's inequality.
\end{proof} 

\subsection{Compactness and $\Gamma$-limit}

We first show a compactness result with respect to dimension-reduction convergence for sequences of functions with equibounded $\e^{2s-3} F_\e$

\begin{theorem}[Non-local dimension-reduction compactness]\label{nlcth}
 Let $u_\e$ be such that $\e^{2s-3} F_\e(u_\e)$ is equibounded. Then there exist $u\in H^1(\omega)$ and a subsequence $u_{\e_j}$ such that,  up to addition of constants, $u_{\e_j}\to u$.
\end{theorem}

\begin{proof} Let $\sigma\in (0,\tfrac12)$ be fixed, and let  $u^\sigma_\e$ be defined in Proposition \ref{prop4}.
By \eqref{est-2} we can apply Lemma \ref{local-lemma} and obtain that we can suppose that $u^\sigma_\e\to u^\sigma$ as $\e\to 0$. 

We have
\begin{equation}\label{claim-1}
\frac1\e\int_{(\Omega')_\e} |u^\sigma_\e-u_\e|dx\le I^1_\e+I^2_\e,
\end{equation}
where
\begin{eqnarray*}
   &&I^1_\e:= \frac1\e \int_{[0,1]\times V} \sum_{{k \in {\mathring{\mathcal I}}^{\e\sigma}_{\rho\overline{\nu}}(\Omega_\e)}} \int_{{\e\sigma}k+{\e\sigma}Q_{\rho\overline{\nu}}} |u_\e(x) - u^{\sigma\e,\rho\overline{\nu}}_\e(\e\sigma k)|\,dx\,d\mu_\e(\rho,\overline{\nu}) \\ \label{L1conv2}
&&I^2_\e:= \frac1\e \int_{[0,1]\times V} \sum_{{k \in {\mathring{\mathcal I}}^{\e\sigma}_{\rho\overline{\nu}}(\Omega_\e)}} \int_{{\e\sigma}k+{\e\sigma}Q_{\rho\overline{\nu}}} |u^{\sigma\e,\rho\overline{\nu}}_\e(x) - u^{\sigma\e,\rho\overline{\nu}}_\e(\e\sigma k)|\,dx\,d\mu_\e(\rho,\overline{\nu}).
\end{eqnarray*}

To give a bound on $I^1_\e$ and $I_\e^2$, we can proceed as in \cite[Section 3.1]{BBD}, using the refined lower estimate as in Lemma \ref{lemma-1}.

Using  a  scaled Poincar\'e--Wirtinger inequality, we have that
\begin{eqnarray*}\label{attempt4} && \hskip-1cm\nonumber
    \sum_{{k \in {\mathring{\mathcal I}}^{\e\sigma}_{\rho\overline{\nu}}(\Omega_\e)}} \int_{{\e\sigma}k+{\e\sigma}Q_{\rho\overline{\nu}}} |u_\e(x) - u_\e^{\rho\overline{\nu}}({\e\sigma}k)|\,dx  \\ \nonumber 
    &\leq& P|{\e\sigma}\rho|^{\frac{d}{2}+s} \sum_{k \in {\mathring{\mathcal I}}^{\e\sigma}_{\rho\overline{\nu}}(\Omega_\e)}\bigg(\int_{({\e\sigma}k+{\e\sigma}Q_{\rho\overline{\nu}})^2}\frac{|u_\e(x)-u_\e(y)|^2}{|x-y|^{d+2s}}\,dxdy\bigg)^{1/2},
    \end{eqnarray*} 
where $P$ is the Poincar\'e--Wirtinger constant for the  $d$-dimensional unit cube. By this estimate, 
using the concavity of the square root and that $\#{\mathring{\mathcal I}}^{\e\sigma}_{\rho\overline{\nu}}(\Omega_\e)\sim \frac{\e|\omega|}{\e^d\rho^{d}r^d}$ we then have 
\begin{eqnarray}\label{I1}\nonumber
I^1_\e\le P \frac1\e \e^{s}\sigma^s\,2^{1-\frac{d}{2}}\e^{\frac12}|\omega|^{\frac12}(1-s) [u_\e]_{H^{s}(\Omega)} \frac{1}{2-s}
\\
= P\sigma^s\e\sqrt{1-s}\,2^{1-\frac{d}{2}}|\omega|^{\frac12} \frac{1}{2-s}  \sqrt{\e^{2s- 3}F_\e(u_\e)}.
\end{eqnarray}

As for $I^2_\e$, we note that
$$
\frac1\e\int_{{\e\sigma}k+{\e\sigma}Q_{\rho\overline{\nu}}} |u^{{\sigma\e},\rho\overline{\nu}}_\e(x) - u_\e^{{\sigma}\e,\rho\overline{\nu}}({\e\sigma}k)|\,dx\le\sigma\rho \sqrt d\int_{{\e\sigma}k+rQ_{\rho\overline{\nu}}} |\nabla u^{{\sigma\e},\rho\overline{\nu}}_\e(x)|\,dx.
$$
This implies, using Lemma \ref{lemma-1}, that
\begin{eqnarray}\label{I2}\nonumber
&&\hskip-1cmI^2_\e\le \sigma\sqrt\e \sqrt{d\, 2^{-d}|\omega|} \Bigl(  \int_{[0,1]\times V} \sum_{k \in {\mathring{\mathcal I}}^{\e\sigma}_{\rho\overline{\nu}}(\Omega_\e)}\int_{\e\sigma k+\e\sigma Q_{\rho\overline{\nu}}} |\nabla u_\e^{\sigma\e,\rho\overline{\nu}}(x)|^2 dx\, d\mu_\e(\rho,\overline{\nu}) \Bigr)^{\frac{1}{2}}\\
     &&\hskip-.5cm\le\e\,  \sigma^{2s-1} \sqrt{\frac{d|\omega|}{2^d{c_d}}} \big(\e^{2s- 3}F_\e(u_\e)\big)^{\frac12}.
\end{eqnarray}
From \eqref{claim-1}, \eqref{I1}, and \eqref{I2} we obtain that $u_\e\to u^\sigma$ in $L^1$. In particular we obtain that $u^\sigma$ is independent of $\sigma$. 
\end{proof}

\begin{theorem}\label{Gammath} Let $s=s_\e\to 1^-$ as $\e\to 0$. 
Then we have
$$
\Gamma\hbox{-}\lim_{\e\to 0} \e^{2s-3} F_\e(u)= c_d\int_\omega |\nabla'u|^2dx'
$$
for all $u\in H^1(\omega)$.
\end{theorem}

\begin{proof} Let $u_\e\to u$, and for fixed $\sigma\in(0,\frac12)$ let $u^\sigma_\e$ be defined by \eqref{sigmae}, so that $u^\sigma_\e\to u$ by the previous theorem. Then by \eqref{est-2} we have
\begin{equation}\label{est-22}
 c_d\liminf_{\e\to 0}\sigma^{2(1-s)}\frac1\e \int_{(\Omega')_\e} |\nabla u^\sigma_\e|^2dx\le \liminf_{\e\to 0}\e^{2s-3} F_\e(u_\e).
\end{equation}
Since $\sigma^{2(1-s)}\to 1$ we then obtain
$$c_d\int_{\omega'}|\nabla 'u|^2dx'\le \liminf_{\e\to 0}\e^{2s-3} F_\e(u_\e)
$$
for all $\omega'$ compactly contained in $\omega$, so that the lower bound follows. 

As for the upper bound, this is proved in Proposition \ref{limsupc2} for $u\in C^2(\overline\omega)$.
For $u\in H^1(\omega)$ it suffices to argue by density. 
\end{proof}

\subsection{Thin-film critical regime at the membrane scaling}
When bulk applied forces are considered, the complete functional to study is of the form
$$
F_{\e,s}(u)-\int_{\Omega_\e} g(x')u(x)\,dx.
$$
Since the second integral scales as $\e$, it may be interesting to reread the previous $\Gamma$-convergence theorem when $F_{\e,s}$ is divided only by $\e$. For this scaling, we then have the following result, which can be interpreted as 
giving 
$$
1-s\sim\frac1{|\log\e|}
$$
as a {\em critical regime}.

\begin{theorem}[membrane scaling $\Gamma$-limit]
Let $s=s_\e$. We can suppose that there exists the limit 
$$
\kappa=\lim_{\e\to 0}  \e^{1-s}.
$$
Note that $\kappa\in[0,1]$.
Then the following statements hold.

{\rm i)} If  $\kappa=1$; that is 
$$
1-s<\!<\frac1{|\log\e|},
$$
then 
$$
\Gamma\hbox{-}\lim_{\e\to0} \frac1\e F_{\e,s}(u)=c_d\int_\omega |\nabla'u|^2dx'
$$
with domain $H^1(\omega)$;

{\rm ii)} if $\kappa\in (0,1)$, which is the case when
$$
1-s\sim\frac1{|\log\e|},
$$
 then 
$$
\Gamma\hbox{-}\lim_{\e\to0} \frac1\e F_{\e,s}(u)=\kappa^2 c_d\int_\omega |\nabla'u|^2dx'
$$
with domain $H^1(\omega)$;

{\rm iii)}  if  $\kappa=1$; that is,
$$
1-s>\!>\frac1{|\log\e|},
$$
then the $\Gamma$-limit of $\frac1\e F_{\e,s}(u)=0$ for all functions $u\in L^2(\omega)$.
 
 Moreover,  if $\kappa\ne0$ the functionals $ \frac1\e F_{\e,s}$ are equi-coercive.
\end{theorem}

\begin{proof}
Claims (i) and (ii) are equivalent to Theorem \ref{Gammath} in the case $\e^{1-s}\to \kappa$. Claim (iii) is an immediate consequence of Proposition \ref{limsupc2} and the density of $H^1(\omega)$ in  $L^2(\omega)$.
\end{proof}

%{\color{red} Vedere se si pu\`o determinare quando il limite \`e $0$ SOLO su  $L^2(\omega)$.}

\section{Fractional thin films}
%We consider the double $\Gamma$-limit as $\e\to 0$ first for fixed $s$ and then as $s\to 1$, and compare it with the limits as $\e$ tends to $0$ `much faster' than $1-s$. 
%
We now examine the dimension-reduction process starting from $H^s$ seminorms on thin films with fixed $s\in(\frac12,1)$. For all $\e>0$ we set
\begin{eqnarray*}
G^s_\e(u) = \int_{\Omega_\e\times\Omega_\e} \frac{|u(x)-u(y)|^2}{|x-y|^{d+2s}} dx\,dy.
\end{eqnarray*}
In this section we examine the behaviour of $G^s_\e$ as $\e\to 0$. To that end, we will compute some $\Gamma$-limits with respect to the reduction-dimension convergence $u_\e\to u$ of the scaled functionals
\begin{eqnarray}\label{Gas}
\frac1{\e^\alpha} G^s_\e(u).
\end{eqnarray}

We note that for $\alpha=3-2s$ the functionals $(1-s)\frac1{\e^\alpha} G^s_\e(u)$ are equicoercive with respect to the convergence $u_\e\to u$,
and converge to $c_d\int_\omega|\nabla' u|^2dx'$. We then have 
\begin{equation}\label{df}
\Gamma\hbox{-}\lim_{s\to 1}(1-s)G_s(u)=c_d\int_\omega|\nabla' u|^2dx',
\end{equation}
where
$$
G_s(u)=\Gamma\hbox{-}\lim_{\e\to 0}\frac1{\e^{3-2s}} G^s_\e(u)
$$
(see \cite{DM,B-LN2}).
As for $G_s$, we note that the compactness Theorem \ref{nlcth} still holds for $s$ fixed since in the use of estimates \eqref{I1} and \eqref{I2} it is not necessary to let $s\to 1$ in its proof, so that the domain of the $\Gamma$-limit is still $H^1(\omega)$, and  by \eqref{est-22} we have
\begin{equation}\label{crit-lb}
\Gamma\hbox{-}\liminf_{\e\to 0}\frac1{\e^{3-2s}} G^s_\e(u)\ge (1-2\sigma)\sigma^{2(1-s)}\frac{c_d}{1-s}\int_\omega|\nabla'u|^2dx'
\end{equation}
for all $\sigma\in(0,\frac12)$. Note that, loosely speaking, in this case the kernels in the Gagliardo seminorms, scaled by $\frac1{\e^{3-2s}}$ act as Bourgain-Brezis-Mironescu kernels. We do not compute this limit, but just note that for $u\in C^2(\overline\omega)$ Proposition \ref{limsupc2} provides also an upper bound in terms of the Lipschitz constant of $u$ and its Dirichlet integral (see also the proof of (i) in Theorem \ref{ftf-th}). We only note that in this special case, since the limit is a quadratic form, integral-representation results using the theory of Dirichlet form suggest that the limit is still a constant (behaving as $\frac{c_d}{1-s}$ as $s\to 1$ by \eqref{df}) times the Dirichlet integral  (see \cite{MR4690560}, where it is also shown that this may not be the case if $p\neq 2$). 

\smallskip
These arguments suggest that $\alpha=3-2s$ is a critical scaling for the convergence of the functionals in \eqref{Gas}. Indeed we have the following theorem.

\begin{theorem}\label{ftf-th}
{\rm(i)} If $\alpha< 3-2s$ then we have 
$$
\Gamma\hbox{-}\lim_{\e\to 0} \frac1{\e^\alpha}G^s_\e(u)=0
$$
for all $u\in L^1(\omega)$;

{\rm(ii)} if  $\alpha> 3-2s$ then we have 
$$
\Gamma\hbox{-}\lim_{\e\to 0} \frac1{\e^\alpha}G^s_\e(u)=\begin{cases}0& \hbox{ if $u$ is constant}\\
+\infty & \hbox{ otherwise}\end{cases}
$$
in $L^1(\omega)$.

\end{theorem}

\begin{proof} (i) We note that for $u\in C^2(\omega)$ this follows from Remark \ref{rem-nd}.
It also suffices to note that 
\begin{eqnarray*}
&&\hskip-1cm\int_{\Omega_\e\times\Omega_\e} \frac{|x'-y'|^2}{|x-y|^{d+2s}} dx\,dy\\
&\le& C\e \int_{B^{d-1}_R\times (0,\e)}\frac{|z'|^2}{|z|^{d+2s}} dz\\
&\le& C\Big(\e \int_{B^{d}_\e}\frac{1}{|z|^{d+2s-2}} dz +\e^2 \int_{B^{d-1}_R\setminus B^{d-1}_\e}
\frac{1}{|z'|^{d+2s-2}} dz'\Big) \\
&\le& C\Big(\e \int_{0}^\e t^{1-2s} dt +\e^2 \int_\e^R\tau^{-2s} d\tau\Big) \\
&\le& C\Big(\e\frac{\e^{2-2s}}{1-s}+\e^2\frac{\e^{1-2s}}{2s-1}\Big)= C\e^{3-2s}\Big(\frac{1}{1-s}+\frac{1}{2s-1}\Big).
\end{eqnarray*}
Hence, for $u=u(x')$ Lipschitz we have
$$
\frac1{\e^\alpha}G^s_\e(u)\le C\e^{3-2s-\alpha}=o(1).
$$
For $u\in L^1(\omega)$ we can then argue by density. 

\goodbreak
\smallskip
(ii) the supercritical case follows from \eqref{crit-lb} by comparison
\end{proof}

\begin{remark}\rm
For $\alpha\ge 2$ the theorem can be alternately proved by using the characterization of constant functions in Proposition \ref{constant}.  By comparison, it suffices to consider the case $\alpha=2$.

Note that
$$
\frac1{\e^2}G^s_\e(u)\ge \int_{(\omega\times (0,1))^2} \frac{|v(x)-v(y)|^2}{|x-y|^{d+2s}} dx\,dy,
$$
so that if $\frac1{\e^2}G^s_\e(u_\e)<+\infty$, then the corresponding sequence $v_\e$ has equi-bounded Gagliardo seminorms, and we can suppose it converges to some $v\in H^s(\omega\times (0,1))$. 

Letting $\e\to 0$ we then have 
$$
\liminf_{\e\to0} \frac1{\e^2}G^s_\e(u_\e)\ge \int_{(\omega\times (0,1))^2} \frac{|v(x)-v(y)|^2}{|x'-y'|^{d+2s}} dx\,dy.
$$
Let $I\subset (0,1)$ be an interval, and let $v_I(x')=\int_I v(x',t)dt$. By Jensen's inequality we then have 
$$
\int_{\omega\times\omega} \frac{|v_I(x')-v_I(y')|^2}{|x'-y'|^{d+2s}} dx'\,dy'<+\infty,
$$
which, since $s>\frac12$, also implies that
$$
\int_{\omega\times\omega} \frac{|v_I(x')-v_I(y')|^2}{|x'-y'|^{(d-1)+2}} dx'\,dy'<+\infty
$$
Hence, by Proposition \ref{constant} applied with $\Omega=\omega$ and $k=d-1$, $v_I$ is a constant. By the arbitrariness of $I$ we obtain that $v=v(x_d)$. If $v$ were not constant, then we would have
$$
\int_{(0,1)\times(0,1)}|v(x_d)-v(y_d)|^2dx_d\,dy_d\ \int_{\omega\times\omega}\frac{1}{|x'-y'|^{d+2s}} dx'\,dy'<+\infty,
$$
which is a contradiction since the second double integral is diverging.\end{remark}

\bibliographystyle{abbrv}

\bibliography{references}

\section*{Appendix: an alternative approach for `thick' thin films}
We propose an equivalent way of defining the dimension-reduction convergence, for which a compactness result can be proven directly in the case when $\e$ is `not too small' with respect to $1-s$.

\begin{definition} Let $u_\e\in L^1(\Omega_\e)$ extended to $\omega\times(-\e,\e)$ by reflection with respect to the hyperplane $x_d=0$, and to the stripe $\omega\times \mathbb R$ by $2\e$-periodicity. We say that $u_\e\to u$ if $u\in L^1_{\rm loc}(\omega)$ and, having set $v(x)=u(x')$ we have $u_\e\to v$ in $L^1(\omega'\times (0,1))$ for all $\omega'\subset\!\subset\omega$.
\end{definition}

Note that if $u_\e$ converges to some $v$ in $L^1(\omega'\times (0,1))$  for all $\omega'\subset\!\subset\omega$, then $v=v(x')$. Indeed, let $\Omega=\omega\times (0,1)$, let $\varphi\in C^\infty_c(\Omega)$ and compute 
$$
\int_\Omega \frac{\partial v}{\partial x_d}\phi \,dx= -\int_\Omega \frac{\partial \phi}{\partial x_d}v \,dx=
-\lim_{t_\e\to 0}\int_\Omega \frac{\phi(x+t_\e e_d)-\phi(x)}{t_\e}v(x) \,dx
$$
$$=
-\lim_{t_\e\to 0}\int_\Omega \frac{v(x-t_\e e_d)-v(x)}{t_\e}\phi(x) \,dx=0,
$$
since we can choose $t_\e\in2\e\mathbb Z$. This shows that $\frac{\partial v}{\partial x_d}=0$ in the sense of distributions, and hence $v=v(x')$. Moreover,
$$
\int_{\omega'\times (0,1)} |u_\e(x)-v(x')|dx\sim \frac1{\e}\int_{\omega'\times (0,\e)} |u_\e(x)-v(x')|dx
$$
$$
= \int_{\omega'\times (0,1)} |v_\e(x)-v(x')|dx,
$$
where $v_\e(x)= u_\e(x',\e x_d)$, so that we recover the definition by scaling.

\smallskip
This convergence is adapted to the energies $F_{\e,s}$. This is proved by a Compactness Theorem stating that if $s=s_\e\to 1^-$ and $u_\e$ is a sequence with equibounded $F_{\e,s}(u_\e)$, then, up to subsequences, $u_\e\to u$ in the sense above, and moreover $u\in H^1(\omega)$. 

\begin{theorem}[Dimension-reduction compactness for thick thin films] Let $\e\to 0^+$, $s=s_\e\to 1^-$, and let $u_\e$ be a sequence such that 
\begin{equation}\label{limitaz}\sup_\e \Big(F_{\e,s}(u_\e)+ \frac1\e\int_{\Omega_\e}|u_\e|^2dx\Big)=:S<+\infty.\end{equation} 
Furthermore we assume that
\begin{equation}\label{11}
\limsup_{\e\to 0} \frac{1-s}{\e^2}<+\infty
\end{equation}
Then, up to subsequences, there exists a function $u\in H^1(\omega)$ such that $u_\e\to u$. \end{theorem}

\begin{remark}\label{re1}\rm Condition \eqref{11} requires that the thickness of the thin film is not too small with respect to $1-s$; that is, there exists $M>0$ such that
\begin{equation}\label{12}
\frac1M\sqrt{1-s}\le \e.
\end{equation}
The same condition appears in \cite{BBD} in order to allow for homogenization.
Note that \eqref{12} implies that  $1\ge \e^{2(1-s)}\ge \frac1{M^{2(1-s)}}(1-s)^{1-s}\to 1$, so that  $\e^{2(1-s)}\to 1$.
\end{remark}

\begin{proof} We want to apply the Bourgain-Brezis and Mironescu result to the ($2\e$-periodic) 
sequence $u_\e$ on $\Omega$. To that end, we need to prove that
\begin{equation}\label{otto}
\sup_\e\ (1-s)\int_{\Omega\times\Omega} \frac{|u_\e(x)-u_\e(y)|^2}{|x-y|^{d+2s}} dx\,dy<+\infty.
\end{equation}
 We first give a bound for 
\begin{eqnarray*}
&&\hskip-1cm(1-s)\int_{\{(x,y)\in \Omega\times\Omega:|x-y|>r\sqrt{1-s}\}} \frac{|u_\e(x)-u_\e(y)|^2}{|x-y|^{d+2s}} dx\,dy
\\&&\le C
(1-s)\int_{\{(x,y)\in \Omega\times\Omega:|x-y|>r\sqrt{1-s}\}} \frac{|u_\e(x)|^2+|u_\e(y)|^2}{|x-y|^{d+2s}} dx\,dy
\\&&\le C
(1-s)\int_{\{(x,y)\in \Omega\times\Omega:|x-y|>r\sqrt{1-s}\}} \frac{|u_\e(x)|^2}{|x-y|^{d+2s}} dx\,dy
\\&&\le C
(1-s)\int_{\Omega} |u_\e(x)|^2\int_{\mathbb R^d\setminus B_{r\sqrt{1-s}}}\frac1{|\xi|^{d+2s}} d\xi\,dx
\\&&\le C
(1-s)\frac1{r^{2s}(1-s)^s}\int_{\Omega} |u_\e(x)|^2\,dx.
\end{eqnarray*}
Hence, we have 
\begin{equation}\label{13}
(1-s)\int_{\{(x,y)\in \Omega\times\Omega:|x-y|>r\e\}} \frac{|u_\e(x)-u_\e(y)|^2}{|x-y|^{d+2s}} dx\,dy\le C\frac1{r^{2s}},
\end{equation}
with the constant $C$ depending only on  $M$ and $S$.

We can now estimate 
$$
(1-s)\int_{\{(x,y)\in \Omega\times\Omega:|x-y|<r\sqrt{1-s}\}} \frac{|u_\e(x)-u_\e(y)|^2}{|x-y|^{d+2s}} dx\,dy\,
$$
$$
=(1-s)\int_{\{(x,y)\in \Omega\times\Omega:|x-y|<r\sqrt{1-s}\}} \frac{|u_\e(x)-u_\e(y)|^2}{|x-y|^{d+2s}} dx\,dy\,
$$
$$
\le(1-s)\sum_{k=0}^{\lfloor1/2\e\rfloor+1}\int_{\omega\times(-\e,\e)+2\e k e_d}
\int_{\{y\in\Omega:|x-y|<r\sqrt{1-s}\}} \frac{|u_\e(x)-u_\e(y)|^2}{|x-y|^{d+2s}} dx\,dy\,
$$
$$
\le(1-s)\sum_{k=0}^{\lfloor1/2\e\rfloor+1}\int_{\omega\times(-\e,\e)+2\e k e_d}
\int_{\{y\in\Omega:|x-y|<rM\e\}} \frac{|u_\e(x)-u_\e(y)|^2}{|x-y|^{d+2s}} dx\,dy\,,
$$
where we have used \eqref{12}. Note now that if $x\in \omega\times(-\e,\e)+2\e k e_d$ then
$$
\{y\in\Omega:|x-y|<rM\e\}\subset \bigcup_{\ell=\lfloor -rM/2\rfloor}^{\lfloor rM/2\rfloor+1}( \omega\times(-\e,\e)+2\e (k+\ell) e_d),
$$
and that {%\color{red} 
if $x,y\in \omega\times(-\e,0)$ and $|x-y|\le rM\e$ then 
$$
|x-y|\le C|x-y+2\e m e_d|
$$
for all $m\in\mathbb Z$.} We then deduce that
\begin{eqnarray*}\nonumber
(1-s)\int_{\{(x,y)\in \Omega\times\Omega:|x-y|<r\sqrt{1-s}\}} 
\frac{|u_\e(x)-u_\e(y)|^2}{|x-y|^{d+2s}} dx\,dy\,
\\ \nonumber
\le(1-s)\Big(\frac1{2\e}+1\Big) (rM+2)\int_{(\omega\times(-\e,\e))^2}
\frac{|u_\e(x)-u_\e(y)|^2}{|x-y|^{d+2s}} dx\,dy.
\end{eqnarray*}
It remain now to observe that
$$
\int_{\omega\times(0,\e)}\int_{\omega\times(-\e,0)}\frac{|u_\e(x)-u_\e(y)|^2}{|x-y|^{d+2s}} dx\,dy
\le \int_{(\omega\times(0,\e)^2)}\frac{|u_\e(x)-u_\e(y)|^2}{|x-y|^{d+2s}} dx\,dy
$$
to deduce that
\begin{equation}\label{15}
(1-s)\int_{\{(x,y)\in \Omega\times\Omega:|x-y|<r\sqrt{1-s}\}} 
\frac{|u_\e(x)-u_\e(y)|^2}{|x-y|^{d+2s}} dx\,dy\,\le CrS.
\end{equation}
From \eqref{8} and \eqref{15} we deduce the validity of \eqref{otto}. Since by \eqref{limitaz} the sequence $u_\e$ is also bounded in $L^2(\Omega)$ we conclude that is it precompact in $L^2(\Omega)$, and that its limits are in $H^1(\Omega)$. The claim then follows by Remark \ref{re1}.
\end{proof} 

\end{document}